\documentclass{amsart}

\usepackage{graphicx}
\usepackage{amsmath}
\usepackage{amssymb}
\usepackage{amsfonts}
\usepackage{tikz}
\usetikzlibrary{tqft}
\usetikzlibrary{shapes}
\usepackage{amscd}
\usepackage{tikz-cd} 
\usepackage[textwidth=6.20in, hcentering]{geometry}

 \newtheorem{theorem}{Theorem}[section]
 \newtheorem{corollary}[theorem]{Corollary}
 
 \newtheorem{proposition}[theorem]{Proposition}
 \newtheorem{remark}[theorem]{Remark}
 \newtheorem{ex}[theorem]{Example}
 \theoremstyle{definition}
 \newtheorem{definition}[theorem]{Definition}
  \newtheorem{notation}[theorem]{Notation}
 \theoremstyle{remark}
 \numberwithin{equation}{subsection}

\begin{document}

\title{Gluing Formulas for Volume Forms on Representation Varieties of Surfaces}

\author{Esma Dirican Erdal}

\address{Deptartment of Mathematics,
\.{I}zmir Institute of Technology,
 35430, \.{I}zmir, Turkey}
 
\email{esmadiricanerdal@gmail.com, esmadirican131@gmail.com}

\date{\today}

\maketitle

\begin{abstract}
Let $\Sigma_{g,n}$ be a compact oriented surface with genus $g\geq 2$ bordered by $n$ circles. Due to Witten, the twisted Reidemeister torsion coincides with a power of the Atiyah-Bott-Goldman-Narasimhan symplectic form on the space of representations of $\pi_1(\Sigma_{g,0})$ in any semi-simple Lie group. In the present paper, we first obtain a multiplicative gluing formula for the twisted Reidemeister torsion of $\Sigma_{g,0}$ in terms of torsions of $\Sigma_{2,2},$ $\Sigma_{2,1},$ and boundary circles $\mathbb{S}^1.$
Then, by using Heusener and Porti's results on $\Sigma_{g,n},$ we show that the symplectic volume form on the representation variety of $\Sigma_{g,0}$ can be expressed as a product of the holomorphic symplectic volume forms on the relative representation varieties of surfaces $\Sigma_{2,1}$ and $\Sigma_{2,2}.$ 

\end{abstract}

\section{Introduction}
\label{intro}
Throughout this paper we denote the compact oriented connected surface of genus $g=2^t=2k\geq 4$ with boundary disjoint union of $n\geq 1$ circles by $\Sigma_{g,n}$ and the closed surface by $\Sigma_{g,0}.$ Let $G$ be the complex reductive group $SL_2(\mathbb{C}).$ It is well-known that there are different decompositions of the surface $\Sigma_{g,0}.$ For the sake of convenience, we decompose this surface into genus $k$ closed oriented Riemann surfaces $\Sigma_{k,0}.$ We restrict our attention to irreducible representations whose stabilizers coincide with the center of $G$. In \cite{DJJM}, such representations are called \emph{good representations}. For closed surfaces  $\Sigma_{g,0},$ we denote the set of conjugacy classes of good representations from the fundamental group of $\Sigma_{g,0}$ to $G$ by $\mathcal{R}^{g}(\pi_1(\Sigma_{g,0}),G).$ For surfaces with boundary $\Sigma_{g,n}$, we consider $\mathcal{R}(\pi_1(\Sigma_{g,n}),\partial(\Sigma_{g,n}),G)_{\rho_0},$ the relative set of conjugacy classes of representations. Let $\mathcal{R}^{g}(\pi_1(\Sigma_{g,n}),\partial(\Sigma_{g,n}),G)_{\rho_0}$ denote the corresponding open subset of good representations. 

In this paper, we consider the good representations $\rho:\pi_1(\Sigma_{g,n})\rightarrow G$ such that
 \begin{itemize}
 \item[($C_1$)] If $n=0$ and $\Sigma_{g,0}=\Sigma_{k,0}^1{\#}\Sigma_{k,0}^2,$ then the restriction $\rho_{|_{\pi_1(\Sigma_{k,1}^i)}}$ is good and $\partial$-regular for each $i=1,2$.
  \item[($C_2$)] If $n\neq 0,$ then $\rho$ is $\partial$-regular and
 \begin{itemize}
 \item for $n=1$ and $\Sigma_{g,1}=\Sigma_{k,1}{\#}\Sigma_{k,0},$ the restriction $\rho_{|_{\pi_1(\Sigma_{k,b})}}$ is good and $\partial$-regular for each $b=1,2.$ 
 \item for $n=2$ and $\Sigma_{g,2}=\Sigma_{k,1}^1{\#}\Sigma_{k,1}^2,$ the restriction $\rho_{|_{\pi_1(\Sigma_{k,b}^i)}}$ is good and 
$\partial$-regular for each $b=1,2$ and $i=1,2.$
 \end{itemize}
 \end{itemize}
 Here, the superscript in the notation $\Sigma_{k,n}^i$ denotes the $i$-th component in the decomposition.
 
 The Cartan-Killing form $\mathcal{B}$ induces two $\mathbb{C}$-valued differential forms on $\mathcal{R}(\pi_1(\Sigma_{g,0}),G);$ a holomorphic volume form $\Omega_{{\Sigma_{g,0}}}$ and the Atiyah-Bott- Goldman -Narasimhan symplectic form $\omega_{_{\Sigma_{g,0}}}.$ The form $\Omega_{{\Sigma_{g,0}}}$ can be given as the twisted Reidemeister torsion of $\Sigma_{g,0}.$ In this paper, we first prove the following theorem which enables us to compute the twisted Reidemeister torsion of $\Sigma_{g,0}$ in terms of torsions of $\Sigma_{2,2},$ $\Sigma_{2,1},$ and boundary circles $\mathbb{S}^1.$
 \begin{theorem}\label{thm:mainthm}
Let $[\rho] \in \mathcal{R}^{g}(\pi_1(\Sigma_{g,0}),G)$ and let $\rho$ satisfy condition $C_1.$
 For a given basis $\mathbf{h}^1_{\Sigma_{g,0}}$ and a fixed basis 
$\mathbf{h}^{j}_{\mathbb{S}^1}$ of $H^1(\Sigma_{g,0};\mathrm{Ad}_\rho)$ and $H^j(\mathbb{S}^1;{\mathrm{Ad}_{\rho_{|_{\mathbb{S}^1}}}}),$ $j=0,1,$ there exist respectively bases $\mathbf{h}^1_{\Sigma_{2,2}^i}$ and $\mathbf{h}^1_{\Sigma_{2,1}^{\mu}}$ of $H^1(\Sigma_{2,2}^i;\mathrm{Ad}_{\rho_{|_{\Sigma_{2,2}}}})$ and $H^1(\Sigma_{2,1}^{\mu};\mathrm{Ad}_{\rho_{|_{\Sigma_{2,1}}^{\mu}}})$ for each $i \in \{1,\ldots,k-2\}$ and  $\mu \in \{1,2\}$ such that the corrective term becomes one and the 
following formula is valid
\begin{eqnarray*}
 \mathbb{T}(\Sigma_{g,0};\rho,\{0,\mathbf{h}^{1}_{\Sigma_{g,0}},0\})&=& 
\prod_{i=1}^{k-2}
\mathbb{T}(\Sigma_{2,2}^i;{\rho_{|_{\Sigma_{2,2}^i}}},\{0,\mathbf{h}^{1}_{\Sigma_{2,2}^i},0\})\\
&& \times\;\prod_{\mu=1}^{2}
\mathbb{T}(\Sigma_{2,1}^{\mu};{\rho_{|_{\Sigma_{2,1}^{\mu}}}},\{0,\mathbf{h}^{1}_{\Sigma_{2,1}^{\mu}},0\})\\
&& \times\; {\mathbb{T}(\mathbb{S}^1;\rho_{|_{\mathbb{S}^1}},
\{\mathbf{h}^{0}_{\mathbb{S}^1},\mathbf{h}^{1}_{\mathbb{S}^1}\})}^{-k+1}.
\end{eqnarray*}
\end{theorem}
 
For surfaces with boundary, the holomorphic volume form $\Omega_{\Sigma_{g,n}}$ is defined on $\mathcal{R}^{g}(\pi_1(\Sigma_{g,n}),G)$ but the holomorphic symplectic volume form $\omega_{\Sigma_{g,n}}$ is defined on the relative representation variety $\mathcal{R}^{g}(\pi_1(\Sigma_{g,n}),\partial(\Sigma_{g,n}),G)_{\rho_0}$ 
(see, \cite{Poritnw,GLD,SLAW}). In \cite{Witten} Witten established a relation between the volume forms $\Omega_{{\Sigma_{g,0}}}$ and $\omega_{_{\Sigma_{g,0}}}.$ Heusener and Porti generalizes the Witten's result to surfaces with boundary \cite{Poritnw}. Considering these results together with Theorem \ref{thm:mainthm}, we obtain the following result which expresses the Atiyah-Bott-Goldman-Narasimhan symplectic form $\omega_{\Sigma_{g,0}}$ on $\mathcal{R}^g(\pi_1(\Sigma_{g,0}),G)$ as the product of holomorphic symplectic volume forms on $\mathcal{R}^g(\pi_1(\Sigma_{g,n}),\partial(\Sigma_{g,n}),G)_{\rho_0}$ and a volume form on $\mathcal{R}(\partial(\Sigma_{g,n}),G).$

\begin{corollary}\label{cor:closed}
 Let $[\rho] \in \mathcal{R}^{g}(\pi_1(\Sigma_{g,0}),G)$ and let $\rho$ satisfy condition $C_1.$ For a given basis $\mathbf{h}^1_{\Sigma_{g,0}}$ and a fixed basis $\mathbf{h}^{j}_{\mathbb{S}^1}$ of $H^1(\Sigma_{g,0};\mathrm{Ad}_\rho)$ and $H^j(\mathbb{S}^1;{\mathrm{Ad}_{\rho_{|_{\mathbb{S}^1}}}}),$ $j=0,1,$ there exist respectively bases $\mathbf{h}^1_{\Sigma_{2,2}^i}$ and $\mathbf{h}^1_{\Sigma_{2,1}^{\mu}}$ of $H^1(\Sigma_{2,2}^i;\mathrm{Ad}_{\rho_{|_{\Sigma_{2,2}}}})$ and $H^1(\Sigma_{2,1}^{\mu};\mathrm{Ad}_{\rho_{|_{\Sigma_{2,1}}^{\mu}}})$ for each $i \in \{1,\ldots,k-2\}$ and  $\mu \in \{1,2\}$ such that the 
following formula holds
\begin{eqnarray*}
 \left|\omega_{\Sigma_{g,0}}^{6k-3}(\wedge \mathbf{h}^1_{\Sigma_{g,0}})\right|&=& 
\mathcal{M}_k\; \prod_{i=1}^{k-2}
\left|\omega_{\Sigma_{2,2}^i}^5(\wedge \mathbf{h}^1_{\Sigma_{2,2}^i})\wedge 
\nu_1^* \wedge \nu_2^* \right| \; \prod_{\mu=1}^{2}
\left|\omega_{\Sigma_{2,1}^{\mu}}^4(\wedge \mathbf{h}^1_{\Sigma_{2,1}^{\mu}})\wedge \nu_1^* \right|.
\end{eqnarray*}
Here, $\mathcal{M}_k=(6k-3)!\; (4!)^{-2} \; (5!)^{2-k} \; {\frac{\nu(\wedge \mathbf{h}^{1}_{\mathbb{S}^1})^{-2k+2}}{<\wedge\mathbf{h}^{1}_{\mathbb{S}^1},\wedge\mathbf{h}^0_{\mathbb{S}^1}>^{-2k+2}}}\in \mathbb{C}.$
\end{corollary}

 
\section{The Representation Variety of Surfaces}
\label{sec:1}
For the surface $\Sigma_{g,n},$ we assume $g=2k\geq 4$ and $n\geq 0.$ By \cite{ALEV}, as $G$ is algebraic, the set of all representations from 
$\pi_1(\Sigma_{g,n})$ to $G$ 
$$R(\pi_1(\Sigma_{g,n}),G)=\mathrm{Hom}(\pi_1(\Sigma_{g,n}),G)$$
  is an affine algebraic set. 
\begin{definition}[\cite{DJJM}]
 A representation $\rho \in R(\pi_1(\Sigma_{g,n}),G)$ is irreducible if the image $\rho(\pi_1(\Sigma_{g,n}))$ is not contained in any proper parabolic subgroup of $G.$ 
 \end{definition}
 Consider the $G$-action on $R(\pi_1(\Sigma_{g,n}),G)$ by conjugation.Then
 \begin{itemize}
 \item[(i)]The stabilizer of $\rho \in R(\pi_1(\Sigma_{g,n}),G)$ is 
 $$\mathrm{Stab}_\rho=\{g\in G| \;g\rho(\gamma)=\rho(\gamma)g, \forall \gamma \in \pi_1(\Sigma_{g,n})\}.$$
 \item[(ii)] The orbit of $\rho \in R(\pi_1(\Sigma_{g,n}),G)$ is
   $$\mathcal{O}(\rho)=\{g\rho(\gamma) g^{-1}\;| \; g \in G, \gamma \in \pi_1(\Sigma_{g,n})\}.$$ 
 \end{itemize}
  
  \begin{proposition}[\cite{DJJM}]
 A representation  $\rho \in R(\pi_1(\Sigma_{g,n}),G)$ is irreducible if and only if the orbit $\mathcal{O}(\rho)$ is closed in  $R(\pi_1(\Sigma_{g,n}),G)$ and 
 $\mathrm{Stab}_\rho$ is finite.
 \end{proposition}
 
  \begin{definition}
 A representation $\rho \in R(\pi_1(\Sigma_{g,n}),G)$ is good if it is irreducible and its stabilizer $\mathrm{Stab}_\rho$ is the center of the group $G.$
  \end{definition}
  

Denote the set of irreducible $G$-representations of the fundamental group $\pi_1(\Sigma_{g,n})$ by 
$R^i(\pi_1(\Sigma_{g,n}),G).$ It is a Zariski open subset of 
$R(\pi_1(\Sigma_{g,n}),G).$ Note that each equivalence class in 
$$R(\pi_1(\Sigma_{g,n}),G)// G$$ contains a unique closed orbit and the orbit of every irreducible representation is closed. Thus, the categorical quotient of 
$R(\pi_1(\Sigma_{g,n}),G)$ restricted to $R^i(\pi_1(\Sigma_{g,n}),G)$ coincides with the set theoretic quotient. We use the notation $\mathcal{R}^{i}(\pi_1(\Sigma_{g,n}),G)$ for the quotient $R^i(\pi_1(\Sigma_{g,n}),G)/ G.$
Since $\pi_1(\Sigma_{g,n})$ is a free group for $n\geq 1$ and it is a surface group for $n=0$, $\mathcal{R}^{i}(\pi_1(\Sigma_{g,n}),G)$ is a manifold.

Let $R^g(\pi_1(\Sigma_{g,n}),G)$ denote the set of all good representations. By \cite{DJJM}, $R^g(\pi_1(\Sigma_{g,n}),G)$ is also a Zariski open subset of $R^i(\pi_1(\Sigma_{g,n}),G),$ and the action of $G$ on $R^g(\pi_1(\Sigma_{g,n}),G)$ is proper. Consider the following quotient
$$\mathcal{R}^{g}(\pi_1(\Sigma_{g,n}),G)=R^g(\pi_1(\Sigma_{g,n}),G)/G .$$
From the above discussion, it becomes an open subset of $\mathcal{R}^{i}(\pi_1(\Sigma_{g,n}),G),$ and it is also a smooth manifold.

The variety of characters is an affine algebraic set defined by its ring of polynomial functions, as the ring of functions on $R(\pi_1(\Sigma_{g,n}),G)$ invariant by conjugation. More precisely, the variety of characters is 
$$X(\pi_1(\Sigma_{g,n}),G) = R(\pi_1(\Sigma_{g,n}),G)//G.$$
The projection $R(\pi_1(\Sigma_{g,n}),G) \rightarrow X(\pi_1(\Sigma_{g,n}),G)$ factors through a surjective map 
$$\mathcal{R}(\pi_1(\Sigma_{g,n}),G) \rightarrow X(\pi_1(\Sigma_{g,n}),G)$$
and thus, for good representations we have :
\begin{proposition}\label{prp1}
The natural map restricts to an injection 
$\mathcal{R}^g(\pi_1(\Sigma_{g,n}),G) \hookrightarrow X(\pi_1(\Sigma_{g,n}),G)$
whose image is a Zariski open subset and a smooth complex manifold.
\end{proposition}
For the proof of Proposition \ref{prp1}, we refer to 
\cite[Section 1]{DJJM}, \cite[Proposition 3.8]{PEN} and \cite{GLD}.

Let $\partial(\Sigma_{g,n})= \mathbb{S}^1_1 \sqcup \cdots \sqcup  \mathbb{S}^1_n$
denote the decomposition in connected components which are circles. We let
 $\partial_i$ denote an element of the fundamental group of the corresponding  circle $\mathbb{S}^1_i$. 

\begin{definition}\cite[Section 4.3]{Kap} Let $\rho_0 \in R(\pi_1(\Sigma_{g,n}),G).$ Then the relative variety of representations $\mathcal{R}(\pi_1(\Sigma_{g,n}),\partial(\Sigma_{g,n}),G)_{\rho_0}$ is given as 
$$
\{[\rho] \in \mathcal{R}(\pi_1(\Sigma_{g,n}),G) | \rho(\partial_i) \in \mathcal{O}(\rho_0(\partial_i)), i= 1,\ldots,n\}.$$ 
 Here, $\mathcal{O}(\rho_0(\partial_i))$ denotes the conjugacy class of 
 $\rho_0(\partial_i).$
 We also denote
$$\mathcal{R}^g(\pi_1(\Sigma_{g,n}), \partial(\Sigma_{g,n}), G)_{\rho_0} = \mathcal{R}(\pi_1(\Sigma_{g,n}), \partial(\Sigma_{g,n}), G)_{\rho_0} \cap \mathcal{R}^{g}(\pi_1(\Sigma_{g,n}), G).$$
\end{definition}

\begin{definition}\cite[Section 3.5]{RS}
Let $g$ be an element of $G.$ If centralizer of $g$ has minimal dimension among centralizers of elements of $G,$ then $g$ is called regular. Equivalently, its conjugacy class $\mathcal{O}(g)$ has maximal dimension.
\end{definition}

\begin{definition}\cite[Section 2.2]{Poritnw}\label{def:bndry1} Let $\rho \in R(\pi_1(\Sigma_{g,n}, G)$ be a representation. If the elements $\rho(\partial_1), \cdots, \rho(\partial_n)$ are regular, then 
$\rho$ is called $\partial$–regular.
\end{definition}

Given a representation $\rho \in R(\pi_1(\Sigma_{g,n}),G),$ $H^{\ast}(\Sigma_{g,n};\mathrm{Ad}_\rho)$ denotes the group cohomology of $\Sigma_{g,n}$ with coefficients in the Lie algebra $\mathfrak{g}$ of $G$ twisted by the homomorphism
$$\pi_1(\Sigma_{g,n})\stackrel{\rho}{\rightarrow} G \stackrel{\mathrm{Ad}}{\rightarrow}  \mathfrak{g},$$
where $\mathrm{Ad}$ is the adjoint representation. Note  that the Lie algebra $\mathfrak{g}$ turns into a $\pi_1(\Sigma_{g,n})$-module via $\mathrm{Ad}\circ \rho.$ If there is no ambiguity, we denote this module by $\mathfrak{g},$ and the coefficients in cohomology by $\mathrm{Ad}_\rho.$ 
\begin{theorem}[\cite{ASikora}]\label{as1}
Let $\Sigma_{g,n}$ be a compact oriented connected surface of genus $g\geq 2$ with $n\geq 0$ boundary components. If $\rho \in R^{g}(\pi_1(\Sigma_{g,n}),G)$ is a good representation, then there is a natural isomorphism
$$ T_{[\rho]}\mathcal{R}^{g}(\pi_1(\Sigma_{g,n}),G) \cong H^1(\Sigma_{g,n};\mathrm{Ad}_\rho).$$
In particular, 
$$\mathrm{dim}(H^1(\Sigma_{g,n};\mathrm{Ad}_\rho))=-\chi(\Sigma_{g,n})\;d.$$
Here, $d$ is the dimension of $G.$
\end{theorem}

 \begin{theorem}[\cite{Goldman,Poritnw}]\label{gold1}
 Let $\rho \in R(\pi_1(\Sigma_{g,n}),G)$ be a good representation.
 If $n=0$ or $n\geq 1$ and $\rho$ is $\partial$-regular, then
 $$H^2(\Sigma_{g,n};{Ad_{\rho}})=H^0(\Sigma_{g,n};{Ad_{\rho}})=0.$$
\end{theorem}

\section{The Twisted Reidemeister Torsion}
\label{sec:2}
Reidemeister torsion was first introduced by Kurt Reidemeister, where he classified (up to PL equivalence) $3$-dimensional lens spaces \cite{Reidemeister}. Generalization of this invariant to other dimensions was given by Franz \cite{Franz} and the topological invariance for manifolds was shown by Kirby-Siebenmann \cite{RCLC}. The Reidemeister torsion is also an invariant of the basis of the homology of a manifold  \cite{Milnor}. Moreover, for closed oriented Riemann surfaces the Reidemeister torsion is equal to the analytic torsion \cite{Cheeger}.

Let us denote the universal covering of $\Sigma_{g,n}$ by 
$\widetilde{\Sigma_{g,n}}$ and the non-degenerate Killing form on $\mathfrak{g}$ by $\mathcal{B}$ which is defined by $\mathcal{B}(A,B) = 4 \;\mathrm{Trace}(AB).$ Let $\rho \in R(\pi_1(\Sigma_{g,n}),G)$ be a representation. Recall that we consider the action of $\pi_1(\Sigma_{g,n})$ on $\mathfrak{g}$ via the adjoint of $\rho.$ Consider a cell decomposition $K$ of $\Sigma_{g,n}$ and denote by $\widetilde{K}$ the lift of $K$ to $\widetilde{\Sigma_{g,n}}$ and let
 $$\mathbb{Z}[\pi_1(\Sigma_{g,n})]=\left\{\sum_{i=1}^p m_i\gamma_i\; ; m_i\in \mathbb{Z},\;
\gamma_i\in \pi_1(\Sigma_{g,n}),\; p\in \mathbb{N}\right\}$$
be the integral group ring. If $C_{\ast}(\widetilde{K}; \mathbb{Z})$ denotes the cellular chain complex on the universal covering, one defines 
$C^{\ast}(K;{\mathrm{Ad}_\rho})$ as the set of
$\mathbb{Z}[\pi_1(\Sigma_{g,n})]$-module homomorphism from
$C_{\ast}(\widetilde{K};\mathbb{Z})$ to $\mathfrak{g}$. More precisely,
 \begin{equation}\label{defchn1}
 C^{\ast}(K;\mathrm{Ad}_{\rho})=\mathrm{Hom}_{_{\mathbb{Z}[\pi_1(\Sigma_{g,n})]}
 }(C_{\ast}(\widetilde{K}; \mathbb{Z}),\mathfrak{g}).
 \end{equation}

Suppose that $\mathcal{A}=\{\mathfrak{a}_i\}_{i=1}^{\dim \mathfrak{g}}$ is a
$\mathcal{B}$-orthonormal $\mathbb{C}$-basis of $\mathfrak{g}.$
For each $p-$cell $e_j^p$ of $K$ we choose a lift $\widetilde{e_j^p}$ to the universal covering $\widetilde{K},$ then
$$\mathbf{c}^p=\{(\widetilde{e_j^p})^{\ast}\otimes\mathfrak{a}_k\}_{pk}$$
is a basis of $C^{p}(K;{\mathrm{Ad}_
\rho})$, called the \textit{geometric basis}. Here, 
$$(\widetilde{e_j^p})^{\ast}\otimes\mathfrak{a}_k:C_{\ast}(\widetilde{K};{\mathrm{Ad}_\rho})\rightarrow \mathfrak{g}$$ is the unique $\mathbb{Z}[\pi_1(\Sigma_{g,n})]$-homomorphism given by $ (\widetilde{e_j^p})^{\ast}\otimes\mathfrak{a}_k(e_i^p)=\delta_{ji}\mathfrak{a}_k.$

Consider the following cochain complex $C^{\ast}=C^{\ast}(K;\mathrm{Ad}_{\rho})$ 
\begin{equation}\label{chaincomplex}
C^{\ast}= \begin{array}{ccc}
 (0 \to C^{0}(K;{\mathrm{Ad}_
\rho}) {\longrightarrow}
C^{1}(K;{\mathrm{Ad}_
\rho}){\longrightarrow}
C^{2}(K;{\mathrm{Ad}_
\rho}){\rightarrow}0).
  \end{array}
\end{equation} For $p \in\{0,1,2\},$ let
$$B^p(C^{\ast})=\mathrm{Im}\{\partial_{p}:C^{p-1}(C^{\ast})\rightarrow C^{p}(C^{\ast})\},$$ 
$$Z^p(C^{\ast})=\mathrm{Ker}\{\partial_{p+1}:C^{p}(C^{\ast})\rightarrow
C^{p+1}(C^{\ast}) \},$$
 and 
 $$H^p(C^{\ast})=Z^p(C^{\ast})/B^p(C^{\ast})$$ be
$p$-th cohomology group of the cochain complex (\ref{chaincomplex}). 
Then there are
the following short exact sequences
\begin{equation}\label{Equation1}
0\longrightarrow Z^p(C^\ast) \stackrel{\imath}{\hookrightarrow} 
C^p(C^{\ast})
\stackrel{\partial_p}{\longrightarrow} B^{p+1}(C^\ast) \longrightarrow 0,
\end{equation}
\begin{equation}\label{Equation2}
0\longrightarrow B^p(C^\ast) \stackrel{\imath}{\hookrightarrow} 
Z^p(C^\ast)
\stackrel{\varphi_p}{\twoheadrightarrow} H^p(C^\ast) \longrightarrow 0.
\end{equation}
Here, "$\hookrightarrow$" and "$\twoheadrightarrow$" are the
inclusion and the natural projection, respectively. 

Assume that $s_p:B^{p+1}(C_{\ast})\rightarrow C^p(C^{\ast})$ and
$\ell_p:H^p(C^{\ast})\rightarrow Z^p(C^{\ast})$ are sections of
$\partial_p:C^p(C^{\ast}) \rightarrow B^{p+1}(C^{\ast})$ and
$\varphi_p:Z^p(C^{\ast})\rightarrow H^p(C^{\ast}),$
respectively. Then we obtain 
\begin{equation}\label{klur2}
C^p(C^{\ast})=s_{p}(B^{p+1}(C^{\ast}))\oplus\ell_p (H^p(C^{\ast}))\oplus
 B^p(C^{\ast}).
\end{equation}
If $\mathbf{h^p}$ and $\mathbf{b^p}$ are respectively bases of 
$H^p(C^{\ast})$ and $B^p(C^{\ast}),$ then, by equation~(\ref{klur2}), the following disjoint union $$s_p(\mathbf{b}^{p+1})\sqcup \ell_p(\mathbf{h}^p)\sqcup
\mathbf{b}^p$$ becomes a new basis for $C^p(C^{\ast}).$

\begin{definition}
The twisted Reidemeister torsion of $\Sigma_{g,n}$ is defined as the following alternating product
$$\mathbb{T}(\Sigma_{g,n};\rho,\{\mathbf{h}^p\}_{p=0}^{2})=\prod_{p=0}^2 \left[s_p(\mathbf{b}^{p+1})\sqcup \ell_p(\mathbf{h}^p)\sqcup
\mathbf{b}^p, \mathbf{c}^p\right]^{(-1)^{(p)}}\in \mathbb{C}^{\ast}\setminus \{\pm 1\}.$$
 Here, $\left[ \mathbf{e}, \mathbf{f}\right]$ is the determinant of
the transition matrix from basis $\mathbf{f}$ to $\mathbf{e}.$
\end{definition}

The twisted Reidemeister torsion does not depend on the cell-decomposition $K,$ the basis $\mathcal{A},$ the lifts $\widetilde{e}^p_j,$ and the conjugacy classes of $\rho.$ Following \cite{Milnor}, it is also independent of the bases $\mathbf{b}^p$ and the sections $s_p, \ell_p.$

\begin{notation} For an ordered basis $\mathbf{e}=\{e_1,\ldots, e_m\}$ of a vector space, denote 
$$\wedge \mathbf{e} = e_1 \wedge \cdots \wedge e_m.$$ 
Since $\wedge \mathbf{e} = [\mathbf{e}, \mathbf{f}](\wedge \mathbf{f}),$ the notation
$[\mathbf{e},\mathbf{f}] = \wedge \mathbf{e} / \wedge \mathbf{f}$
is often used in the literature (cf. \cite{Milnor2}).
\end{notation} 

The Mayer-Vietoris sequence is one of the useful tools to compute the twisted Reidemeister torsion. More precisely, assume that $X$ is a compact CW-complex with subcomplexes $X_1,$ $X_2 \subset X$ so that $X=X_1\cup X_2$ and $Y=X_1\cap X_2.$ Let $Y_1,\ldots,Y_k$ be the connected components of $Y.$ For $\nu= 1,2,$ consider the inclusions 
$$Y \overset{i_\nu} {\hookrightarrow} X_\nu \overset{j_\nu} {\hookrightarrow}X.$$  For $\mu \in \{1,\ldots,k\},$ let $\rho:\pi_1(X)\rightarrow G$ be a representation with the restrictions 
$\rho_{_{|_{X_{\nu}}}}:\pi_1(X_\nu)\rightarrow G,$ 
$\rho_{_{|_{Y_{\mu}}}}:\pi_1(Y_\mu)\rightarrow G.$ Then there is a Mayer-Vietoris long exact sequence in cohomology with twisted coefficients 
\begin{eqnarray}\label{es12es}
& &\mathcal{H}^{\ast}:\cdots\longrightarrow H^i(X;\mathrm{Ad}_\rho) \longrightarrow H^i(X_1;\mathrm{Ad}_{\rho_{_{|_{X_{1}}}}})\oplus H^i(X_2;\mathrm{Ad}_{\rho_{_{|_{X_{2}}}}}) \nonumber \\ 
&&\quad \quad\quad \quad\quad\quad\quad \quad\; \; \tikz\draw[->,rounded corners,nodes={asymmetrical rectangle}](5.10,0.4)--(5.10,0)--(1,0)--(1,-0.4)node[yshift=4.0ex,xshift=12.0ex] {}; \nonumber\\ 
&&\quad\quad\quad\quad\quad \underset {{\mu}}{\oplus} 
H^i(Y_{\mu};\mathrm{Ad}_{\rho_{_{|_{Y_{\mu}}}}}) 
\longrightarrow H^{i+1}(X;\mathrm{Ad}_\rho)\longrightarrow \cdots
 \end{eqnarray}
Choose a basis for each of these cohomology groups $\mathbf{h}_X^i$ for 
$H^i(X;\mathrm{Ad}_\rho),$ $\mathbf{h}^i_{X_1}$ for $H^i(X_1;\mathrm{Ad}_{\rho_{_{|_{X_{1}}}}}),$ $\mathbf{h}_{X_2}^i$ for $H^i(X_2;\mathrm{Ad}_{\rho_{_{|_{X_{2}}}}}),$ and $\mathbf{h}_{Y_{\mu}}^i$ for 
$H^i(Y_{\mu};\mathrm{Ad}_{\rho_{_{|_{Y_{\mu}}}}}).$ 
The long exact sequence (\ref{es12es}) can be viewed as a cocomplex. 
\begin{notation}
The torsion $\mathbb{T}(\mathcal{H}^{\ast},\{\mathbf{h}^{\ast \ast}\})$ is called the \textit{corrective term} and the same terminology is also used in \cite{BorgStefa}.
\end{notation}
 Following theorem gives a multiplicative gluing formula.
\begin{theorem}\label{prt1}
For a compact CW-complex $X =X_1\cup X_2$ with $Y=X_1\cap X_2,$ if $Y_1,\ldots,Y_k$ are the connected components of $Y$ and 
$\rho:\pi_1(X)\rightarrow G$ is a representation, then by choosing basis in cohomology the following formula holds 
\begin{eqnarray*}
 \mathbb{T}(X_1;\rho_{_{|_{X_{1}}}},\{\mathbf{h}_{X_1}^i\})\; 
 \mathbb{T}(X_2;\rho_{_{|_{X_{2}}}},\{\mathbf{h}_{X_2}^i\})&=&\mathbb{T}(X;\rho,\{\mathbf{h}_X^i\}) \; \prod_{\mu=1}^k\mathbb{T}(Y_{\mu}; \rho_{_{|_{Y_{\mu}}}},\{\mathbf{h}_{Y_{\mu}}^i\})\\
& & \times \; \mathbb{T}(\mathcal{H}^{\ast}, \{\mathbf{h}^{\ast \ast}\}).
\end{eqnarray*}
\end{theorem}
The proof can be found in \cite{Milnor} or in \cite{Porti2}. For further information and the detailed proof, the reader is also referred to \cite{Milnor2,Porti,Turaev,Witten} and the references therein.

\section{The Volume Forms on Representation Varieties of Surfaces}
\label{sec:3}
\subsection{The holomorphic volume form on $\mathcal{R}^g(\pi_1(\Sigma_{g,n}),G)$}
\label{sec:4}
It is obtained from Theorem \ref{as1} that the tangent space of $\mathcal{R}^{g}(\pi_1(\Sigma_{g,n}),G)$ at $[\rho]$ can be identified with $H^1(\Sigma_{g,n};\mathrm{Ad}_\rho).$ Hence, there is a natural holomorphic volume form on $H^1(\Sigma_{g,n};\mathrm{Ad}_\rho)$ :
$$\Omega_{_{\Sigma_{g,n}}}(\wedge \mathbf{h}^1_{\Sigma_{g,n}}) = \pm \mathbb{T}(\Sigma_{g,n}; \rho, 
\{0,\mathbf{h}^1_{\Sigma_{g,n}},0\}),$$
where $\mathbf{h}^1_{\Sigma_{g,n}}$ is a basis for $H^1(\Sigma_{g,n};\mathrm{Ad}_\rho).$ 


\subsection{The symplectic volume form on $\mathcal{R}^g(\pi_1(\Sigma_{g,0}),G)$}
\label{sec:5}

Let $\rho:\pi_1(\Sigma_{g,0})\rightarrow G$ be a representation and let $K$ be a cell decomposition of $\Sigma_{g,0}$. The twisted chains can be associated as
$$C_{\ast}(K;{\mathrm{Ad}_\rho})=\displaystyle
C_{\ast}(\widetilde{K};\mathbb{Z})\displaystyle\otimes \mathfrak{g}
/\sim ,$$
 where, $\sigma\otimes t \sim
\gamma\cdot\sigma\otimes\gamma\cdot t,\forall \gamma\in \pi_1(\Sigma_{g,0})$, the action of $\pi_1(\Sigma_{g,0})$ on $\widetilde{\Sigma_{g,0}} $ is the deck
transformation, and the action of $\pi_1(\Sigma_{g,0})$ on $\mathfrak{g}$ is
the adjoint action. Recall that the twisted cochains are given as 
$$C^*(K;{\mathrm{Ad}_
\rho})=\mathrm{Hom}_{\mathbb{Z}[\pi_1(\Sigma_{g,0})]}(C_*(\widetilde{K};\mathbb{Z}),\mathfrak{g}).$$

\begin{definition}
The Kronecker pairing $\left\langle \cdot ,\cdot \right\rangle :C^{i}\left( K;{\mathrm{Ad}_
\rho}\right) \times C_{i}\left( K;{\mathrm{Ad}_
\rho}\right)   \longrightarrow  \mathbb{C}$
is defined by
$\left<\theta,\sigma{\otimes}_{\phi}t\right>=\mathcal{B}\left(t,\theta\left(\sigma\right)\right),$
where $\mathcal{B}$ denotes the Cartan-Killing form.
\end{definition}

Clearly, the pairing can be extended to the following pairing
$$\left\langle \cdot ,\cdot \right\rangle   :  H^{i}\left( \Sigma_{g,0} ;{\mathrm{Ad}_
\rho}\right) \times H_{i}\left( \Sigma_{g,0} ;{\mathrm{Ad}_
\rho}\right)   \longrightarrow   \mathbb{C}.$$

\begin{definition}
The cup product $\smile_{\mathcal{B}}:C^i(K;{\mathrm{Ad}_
\rho})\times C^{j
}(K;{\mathrm{Ad}_
\rho})\to
C^{i+{j }}(\widetilde{\Sigma_{g,0}};\mathbb{C})$
is defined by 
 $$(\theta_i\smile_{\mathcal{B}}\theta_{j })(\sigma_{i+{j }})=
\mathcal{B}(\theta_i((\sigma_{i+{j }})_\mathrm{front}),\theta_{j
}((\sigma_{i+{j }})_\mathrm{back})).$$
 Here, $\theta_p:C^p(K;{\mathrm{Ad}_
\rho})\rightarrow \mathfrak{g}$ is a $\mathbb{Z}[\pi_1(\Sigma_{g,0})]$-module homomorphism.
\end{definition}

 Note that  $\smile_{\mathcal{B}}$ can
be extended to cohomologies
$$ \smile_{\mathcal{B}}:H^i(\Sigma_{g,0};{\mathrm{Ad}_
\rho})   \times   H^{j}(\Sigma_{g,0};{\mathrm{Ad}_
\rho})   \to   H^{i+j}(\Sigma_{g,0};\mathbb{C}).
$$

Let $K'$ be the dual cell decomposition of $\Sigma_{g,0}$ associated
to the cell decomposition $K$. Assume also that cells $\sigma\in
K,$ $\sigma'\in K'$ meet at most once. This assumption is not loss
of generality since the Reidemeister torsion is invariant under subdivision. Let $c'_p$ be a basis of
$C_p(\widetilde{K'};\mathbb{Z})$ associated to the basis $c_p$ of
$C_p(\widetilde{K};\mathbb{Z}),$ and let
$\mathbf{c}'_p=c'_p\otimes_{\rho} \mathcal{A}$ be the basis for
$C_p(K';{\mathrm{Ad}_\rho}),$ where $\mathcal{A}$ is a $\mathcal{B}$-orthonormal basis of $\mathfrak{g}.$

\begin{definition}
The intersection form 
$
(\cdot,\cdot)_{i,2-i}:C_i(K;{\mathrm{Ad}_
\rho})\times C_{2-i}(K';{\mathrm{Ad}_
\rho})\to\mathbb{C}
$ is defined by $$(\sigma_1\otimes
t_1,\sigma_2\otimes t_2)_{i,2-i}=\sum_{\gamma\in\pi_1(\Sigma_{g,0})}
\sigma_1.(\gamma\cdot\sigma_2)\; \mathcal{B}(t_1,\gamma\cdot t_2),$$ 
where $``."$ is the intersection number pairing. Clearly, $``."$ is
compatible with the usual boundary operator and thus
$(\cdot,\cdot)_{i,2-i}$ are $\partial$-compatible. It is also
anti-symmetric since the pairing
$``."$ is anti-symmetric and $\mathcal{B}$ is invariant under adjoint
action.
\end{definition}


The intersection form can be naturally extended to twisted homologies. 
So we have the non-degenerate form
\begin{equation}\label{intersection formH}
(\cdot,\cdot)_{i,2-i}:H_i(\Sigma_{g,0};{\mathrm{Ad}_
\rho})\times H_{2-i}(\Sigma_{g,0};{\mathrm{Ad}_
\rho})\to \mathbb{C}.
\end{equation}

By the isomorphisms induced by the Kronecker pairing and the
intersection form, we have the following Poincar\'{e} duality isomorphisms
$$
H^{2-i}(\Sigma_{g,0};{\mathrm{Ad}_
\rho})\stackrel{\scriptsize{\mbox{Kronecker pairing}}}{\cong}
H_{2-i}(\Sigma_{g,0};{\mathrm{Ad}_
\rho})^{\ast}\stackrel{\scriptsize{\mbox{Intersection form}}}{\cong}H_i(\Sigma_{g,0};{\mathrm{Ad}_\rho}) .$$
For a good representation $\rho \in R^{g}(\pi_1(\Sigma_{g,0}),G),$ by Theorem \ref{gold1}, we get
$$H_0(\Sigma_{g,0};{\mathrm{Ad}_\rho})=H_2(\Sigma_{g,0};{\mathrm{Ad}_
\rho})=0,$$
$$H^0(\Sigma_{g,0};{\mathrm{Ad}_
\rho})=H^2(\Sigma_{g,0};{\mathrm{Ad}_\rho})=0.$$
The following composition
 $$ \omega_{_{\Sigma_{g,0}}} :
H^1(\Sigma_{g,0};{\mathrm{Ad}_
\rho})\times
H^1(\Sigma_{g,0};{\mathrm{Ad}_
\rho})\stackrel{\smile_B}{\longrightarrow}
H^2(\Sigma_{g,0};\mathbb{C}){\longrightarrow}\mathbb{C}$$
is known as the \textit{Atiyah–Bott–Goldman–Narasimhan symplectic form} for the Lie group $G,$ where the second map sends the fundamental class of 
$\Sigma_{g,0}$ to $1 \in \mathbb{C}.$

In the following theorem, Witten proves that the twisted Reidemister torsion of 
$\Sigma_{g,0}$ equals a power of the Atiyah–Bott–Goldman–Narasimhan symplectic form on $\mathcal{R}^{g}(\pi_1(\Sigma_{g,0}),G).$ 

 \begin{theorem}[\cite{Witten}]\label{wittheo}
 If $\rho \in R^{g}(\pi_1(\Sigma_{g,0}),G)$ is a good representation and $\mathbf{h}^1_{\Sigma_{g,0}}$ is the basis of $H^1(\Sigma_{g,0};{\mathrm{Ad}_\rho}),$ then 
$$|\Omega_{_{\Sigma_{g,0}}}(\wedge \mathbf{h}^1_{\Sigma_{g,0}})|=\left|\mathbb{T}(\Sigma_{g,0},\{0,\mathbf{h}^1_{\Sigma_{g,0}},0\})\right|= \frac{\left|\omega_{_{\Sigma_{g,0}}}^m(\wedge \mathbf{h}^1_{\Sigma_{g,0}})\right|} {m!},$$
 where $\mathrm{dim}(H^1(\Sigma_{g,0};{\mathrm{Ad}_\rho}))=2m.$
 \end{theorem}

 \subsection{The holomorphic symplectic volume form on $\mathcal{R}^g(\pi_1(\Sigma_{g,n}),\partial(\Sigma_{g,n}),G)_{\rho_0}$}
\label{sec:6}


In order to obtain a holomorphic symplectic volume form on $\mathcal{R}(\partial(\Sigma_{g,n}),G),$ Heusener and Porti identified the representation variety of the circle $\mathbb{S}^1$ with $G,$ by mapping each representation to the image of a fixed generator of $\pi_1(\mathbb{S}^1).$ They restricted to representations which map the generators of $\pi_1(\mathbb{S}^1)$ to regular elements \cite{Poritnw}.

\begin{definition}
A representation $\rho:\pi_1(\mathbb{S}^1)\rightarrow G$ is called regular if the image of the generator of $\pi_1(\mathbb{S}^1)$ is a regular element of $G.$
\end{definition}

 Let $\mathbf{r}$ be the rank of simply-connected $G.$ Then Steinberg’s Theorem \cite{RS2} gives the following result
$$\mathcal{R}^{\mathrm{reg}}(\pi_1(\mathbb{S}^1),G)\cong G^{\mathrm{reg}}/G \cong \mathbb{C}^\mathbf{r}.$$
By using the above isomorphisms, the following proposition gives a natural isomorphism as well.
 \begin{proposition}\cite[Corollary 2.12]{Poritnw} \label{prt23}
 If $G$ is simply-connected, then the Steinberg map induces a natural isomorphism
  $$H^1(\mathbb{S}^1;{\mathrm{Ad}_{\rho}})\cong T_{[\rho]}\mathcal{R}^{\mathrm{reg}}(\pi_1(\mathbb{S}^1),G)\cong \mathbb{C}^\mathbf{r}.$$
 \end{proposition}

In \cite{Poritnw}, it was introduced a holomorphic volume form on the space of representations of the circle. Precisely,
\begin{definition}\label{defcrl1}(\cite{Poritnw})
Let $\rho:\pi_1(\mathbb{S}^1)\rightarrow G$ be a regular representation. For any basis $\mathbf{h}^{1}_{\mathbb{S}^1}$ of $H^1(\mathbb{S}^1;{\mathrm{Ad}_{\rho}}),$ the form $\nu=\overset{\mathbf{r}}{\wedge} H^1(\mathbb{S}^1;{\mathrm{Ad}_{\rho}})\rightarrow \mathbb{C}$ is defined by the following formula
\begin{equation}\label{eq:def1}
\nu(\wedge \mathbf{h}^{1}_{\mathbb{S}^1})= \pm \sqrt{ {\mathbb{T}(\mathbb{S}^1;\rho,
\{\mathbf{h}^{0}_{\mathbb{S}^1},\mathbf{h}^{1}_{\mathbb{S}^1}\})}} <\wedge\mathbf{h}^{1}_{\mathbb{S}^1},\wedge\mathbf{h}^0_{\mathbb{S}^1} >.
\end{equation}      
Here, $<\cdot,\cdot>$ is the duality pairing $H^1(\mathbb{S}^1;{\mathrm{Ad}_{\rho}})\times H^0(\mathbb{S}^1;{\mathrm{Ad}_{\rho}}) \rightarrow \mathbb{C}$ and the value $\nu(\wedge \mathbf{h}^{1}_{\mathbb{S}^1}) \in \mathbb{C}$ is independent of $\mathbf{h}^{1}_{\mathbb{S}^1}.$
\end{definition}

Let $\rho_0$ be a good and $\partial$–regular representation. Then the tangent space of $\mathcal{R}^g(\pi_1(\Sigma_{g,n}), \partial(\Sigma_{g,n}), G)_{\rho_0}$ is the kernel of the map 
$i: H^1(\Sigma_{g,n};\mathrm{Ad}_\rho) \rightarrow H^1(\partial(\Sigma_{g,n});\mathrm{Ad}_\rho)$ induced by inclusion (see \cite[Proposition 2.10]{Poritnw}). 
The long exact sequence in cohomology of the pair is 
\begin{eqnarray*}
&& 0 \rightarrow H^0(\partial(\Sigma_{g,n});\mathrm{Ad}_\rho)\stackrel{\beta} {\rightarrow} H^1(\Sigma_{g,n},\partial(\Sigma_{g,n});\mathrm{Ad}_\rho) \\
&&\quad \quad \quad \quad \quad \quad \;\;\tikz\draw[->,rounded corners](5.25,0.4)--(5.25,0)--(1,0)--(1,-0.4) node[yshift=4.5ex,xshift=15.0ex] {j}; \nonumber\\ 
&& \quad \quad \quad H^1(\Sigma_{g,n};\mathrm{Ad}_\rho) \stackrel{i} {\rightarrow} H^1(\partial(\Sigma_{g,n});\mathrm{Ad}_\rho) \rightarrow 0.
\end{eqnarray*}
For $a,b \in \;\mathrm{ker}(i),$ we define
\begin{equation}\label{frm2}
\omega_{\Sigma_{g,n}}(a,b)=<\tilde a,b>=<a,\tilde b>.
\end{equation}
Here, $a, \tilde{b} \in H^1(\Sigma_{g,n}, \partial(\Sigma_{g,n}); Ad_{\rho})$ satisfy $j(\tilde{a})=a,$ 
$j(\tilde{b})=b.$ The form $\omega_{\Sigma_{g,n}}$ is symplectic on the relative character variety $\mathcal{R}^g(\pi_1(\Sigma_{g,n}),\partial(\Sigma_{g,n}),G)_{\rho_0}$ by 
\cite{GLD,KGJH,SLAW}.

Set $\mathbf{d}=\mathrm{dim}(G)$ and $\mathbf{r}=\mathrm{rank}(G).$ Let $\mathbf{n}$ denote the number of the boundary components of $\Sigma_{g,n},$ where $\partial(\Sigma_{g,n})= \mathbb{S}^1_1 \sqcup \cdots \sqcup  \mathbb{S}^1_n.$ Then the following result is obtained from Definition \ref{def:bndry1}.
\begin{remark} \cite{Poritnw} \label{rem:bndry1}
For any $\partial$-regular representation $\rho:\pi_1(\Sigma_{g,n})\rightarrow G,$ the restriction $\rho_{|_{\pi_1(\mathbb{S}^1_i)}}:\pi_1(\mathbb{S}^1_i)\rightarrow G$ on each boundary component $\mathbb{S}^1_i$ is regular.
\end{remark}
 Let us denote the form related to the restriction by $\nu_i:\overset{\mathbf{r}}{\wedge} H^1(\mathbb{S}^1_i;\mathrm{Ad}_\rho)\rightarrow \mathbb{C}$ as in equation (\ref{eq:def1}) and the value $\nu_i(\wedge \mathbf{h}^{1}_{\mathbb{S}_i^1}) \in \mathbb{C}$ by $\nu_i^*.$ The following result generalizes Theorem \ref{wittheo} to surfaces with boundary. Namely, it expresses the volume form on 
 $\mathcal{R}^g(\pi_1(\Sigma_{g,n}),G)$ as the product of a volume form on $\mathcal{R}^g(\pi_1(\Sigma_{g,n}),\partial(\Sigma_{g,n}),G)_{\rho_0}$ and a volume form on $\mathcal{R}(\partial(\Sigma_{g,n}),G).$

\begin{theorem}\cite[Theorem 1.3]{Poritnw}\label{prtitheo}
Let $\rho_0:\pi_1(\Sigma_{g,n})\rightarrow G$ be a good and $\partial$-regular representation. Let $\mathbf{h}^1_{\Sigma_{g,n}}$ be a given basis of $H^1(\Sigma_{g,n};{\mathrm{Ad}_\rho}).$ Then, on $T_{[\rho_0]}\mathcal{R}^{g}(\pi_1(\Sigma_{g,n}),G),$ there is a volume form given as follows
$$|\Omega_{_{\Sigma_{g,n}}}(\wedge \mathbf{h}^1_{\Sigma_{g,n}})|=\left|\mathbb{T}(\Sigma_{g,n},\{0,\mathbf{h}^1_{\Sigma_{g,n}},0\})\right|= \frac{\left|\omega_{\Sigma_{g,n}}^m(\wedge \mathbf{h}^1_{\Sigma_{g,n}})\wedge \nu_1^* \wedge\ldots \wedge \nu_n^* \right|} {m!},$$
 where $m=\frac{1}{2}\mathrm{dim}(\mathcal{R}^g(\pi_1(\Sigma_{g,n}),\partial(\Sigma_{g,n}),G)_{\rho_0})=\frac{1}{2}(-\chi(\Sigma_{g,n})\mathbf{d}-\mathbf{nr}).$ 
 \end{theorem}

\section{Main Results}
\label{sec:7}

Along this section, for surfaces $\Sigma_{g,n},$ we assume that $g=2^t, k=g/2, t\geq 2$ and $\rho_0:\pi_1(\Sigma_{g,n})\rightarrow G$ is good and $\partial$-regular for each $n\geq 1.$
\subsection{Proof of Theorem  \ref{thm:mainthm}}
\label{sec:8}
To prove Theorem \ref{thm:mainthm}, we obtain the following auxiliary results.
\begin{theorem}\label{mainth1} 
Let $[\rho] \in \mathcal{R}^{g}(\pi_1(\Sigma_{g,2}),\partial(\Sigma_{g,2}),G)_{\rho_0}$ and let $\rho$ satisfy condition $C_2$ for the connected sum 
$$\Sigma_{g,2}=\Sigma_{k,1}^1 \# \Sigma_{k,1}^2.$$
 For given bases $\mathbf{h}^1_{\Sigma_{g,2}}$ and $\mathbf{h}^{i}_{\mathbb{S}^1}$ of $H^1(\Sigma_{g,2};\mathrm{Ad}_\rho)$ and $H^i(\mathbb{S}^1;{\mathrm{Ad}_{\rho_{|_{\mathbb{S}^1}}}}),$ $i=0,1,$ there exist bases $\mathbf{h}^1_{\Sigma_{k,2}^1}$ and $\mathbf{h}^1_{\Sigma_{k,2}^2}$ of $H^1(\Sigma_{k,2}^1;\mathrm{Ad}_{\rho_{|_{\Sigma_{k,2}^1}}})$ and $H^1(\Sigma_{k,2}^2;\mathrm{Ad}_{\rho_{|_{\Sigma_{k,2}^2}}}),$ respectively, such that the corrective term becomes one and the 
following formula is valid
\begin{eqnarray}\label{eq1mn1}
 \mathbb{T}(\Sigma_{g,2};\rho,\{0,\mathbf{h}^{1}_{\Sigma_{g,2}},0\})&=&\prod_{i=1}^{2}\mathbb{T}(\Sigma_{k,2}^i;{\rho_{|_{\Sigma_{k,2}^i}}},\{0,\mathbf{h}^{1}_{\Sigma_{k,2}^i},0\})\nonumber \\
& & \times\;{\mathbb{T}(\mathbb{S}^1;\rho_{|_{\mathbb{S}^1}},\{\mathbf{h}^{0}_{\mathbb{S}^1},\mathbf{h}^{1}_{\mathbb{S}^1}\})}^{-1}.
\end{eqnarray}
For $\Sigma_{g,0}=\Sigma_{k,0}^1 \# \Sigma_{k,0}^2,$ we just need the condition $C_1$ with $[\rho] \in \mathcal{R}^{g}(\pi_1(\Sigma_{g,0}),G).$ Moreover, the following formula holds
\begin{eqnarray}\label{eq2mn1}
 \mathbb{T}(\Sigma_{g,0};\rho,\{0,\mathbf{h}^{1}_{\Sigma_{g,0}},0\})&=&\prod_{i=1}^{2}\mathbb{T}(\Sigma_{k,1}^i;{\rho_{|_{\Sigma_{k,1}^i}}},\{0,\mathbf{h}^{1}_{\Sigma_{k,1}^i},0\})\nonumber \\
& & \times\;{\mathbb{T}(\mathbb{S}^1;\rho_{|_{\mathbb{S}^1}},\{\mathbf{h}^{0}_{\mathbb{S}^1},\mathbf{h}^{1}_{\mathbb{S}^1}\})}^{-1}.
\end{eqnarray}
\end{theorem}

\begin{proof} We only give the proof of the first case since the other can be obtained by using the same arguments. By condition $C_2,$ the restrictions 
$\rho_{|_{\pi_1(\Sigma_{k,2}^1)}},$ $\rho_{|_{\pi_1(\Sigma_{k,2}^2)}}$ are good and $\partial$-regular. From Remark \ref{rem:bndry1} it follows that the restriction $\rho_{|_{\mathbb{S}^1}}$ is regular. Consider the following short exact sequence of the twisted cochain complexes
\begin{eqnarray}\label{seq3a}
&& 0\to C^{\ast}(\Sigma_{g,2};\mathrm{Ad}_\rho) \rightarrow 
C^{\ast}(\Sigma_{k,2}^1;\mathrm{Ad}_{\rho_{|_{\Sigma_{k,2}^1}}})\oplus
C^{\ast}(\Sigma_{k,2}^2;\mathrm{Ad}_{\rho_{|_{\Sigma_{k,2}^2}}})\nonumber\\
 &&\quad \quad \quad\quad  \quad \tikz\draw[->,rounded corners](5.60,0.4)--(5.60,0)--(1,0)--(1,-0.4) node[yshift=4.5ex,xshift=15.0ex] {}; \nonumber\\ 
&&\quad \quad C^{\ast}(\mathbb{S}^1;{\mathrm{Ad}_{\rho_{|_{\mathbb{S}^1}}}}) \to 0.
\end{eqnarray}

By Theorem \ref{as1}, Theorem \ref{gold1} and Proposition \ref{prt23}, we have 
\begin{eqnarray}\label{eqy1a}
\left\{ \begin{array}{*{20}{l}}
&& \mathrm{dim}(H^1(\mathbb{S}^1;{\mathrm{Ad}_{\rho_{|_{\mathbb{S}^1}}}}))=\mathrm{dim}(H^0(\mathbb{S}^1;{\mathrm{Ad}_{\rho_{|_{\mathbb{S}
^1}}}}))=1,\\ 
&& \mathrm{dim}(H^1(\Sigma_{g,2};\mathrm{Ad}_\rho))=12k,\\ 
&& \mathrm{dim}(H^1(\Sigma_{k,2}^i;\mathrm{Ad}_{\rho_{|_{\Sigma_{k,2}}}}))=6k. 
 \end{array} \right.
\end{eqnarray}

By Theorem \ref{gold1} and equation (\ref{eqy1a}), the Mayer-Vietoris long exact sequence associated with the short exact sequence (\ref{seq3a}) becomes
\begin{eqnarray}\label{long3a}
 \nonumber &\mathcal{H}^{\ast}:& 0\longrightarrow H^0(\mathbb{S}^1;{\mathrm{Ad}_{\rho_{|_{\mathbb{S}^1}}}}) \stackrel{\alpha_0}{\longrightarrow} H^1(\Sigma_{g,2};\mathrm{Ad}_\rho)\\
 &&\quad \quad \quad \quad \quad\quad \quad \quad \; \tikz\draw[->,rounded corners](3.30,0.4)--(3.30,0)--(1,0)--(1,-0.4) node[yshift=4.5ex,xshift=8.0ex] {$\beta_0$}; \nonumber\\ 
 && H^1(\Sigma_{k,2}^1;\mathrm{Ad}_{\rho_{|_{\Sigma_{k,2}^1}}})\oplus H^1(\Sigma_{k,2}^2;\mathrm{Ad}_{\rho_{|_{\Sigma_{k,2}^2}}}) \nonumber\\
 &&\quad \quad \quad   \tikz\draw[->,rounded corners](2.90,0.4)--(2.90,0)--(1,0)--(1,-0.4) node[yshift=4.5ex,xshift=5.0ex] {$\partial$}; \nonumber\\ 
&& H^1(\mathbb{S}^1;{\mathrm{Ad}_{\rho_{|_{\mathbb{S}^1}}}}) \stackrel{\alpha_1} {\longrightarrow} 0 .
  \end{eqnarray}

First, we denote the vector spaces in $\mathcal{H}^{\ast}$ by 
 $C^p(\mathcal{H}^{\ast}),$ $p\in \{0,1,2,3\}.$ For each $p,$ the exactness of the sequence (\ref{long3a}) yields the following short exact sequence  
\begin{equation*}\label{wes1}
0\to B^{p}(\mathcal{H}^{\ast}) \hookrightarrow C^p(\mathcal{H}^{\ast})
 \twoheadrightarrow B^{p+1}(\mathcal{H}^{\ast}) \to 0.
\end{equation*}
For all $p,$  considering the isomorphism $s_{p}:B^{p+1}(\mathcal{H}^{\ast}) \rightarrow s_{p}(B^{p+1}(\mathcal{H}^{\ast}))\subseteq C^p(\mathcal{H}^{\ast})$ obtained by the First Isomorphism Theorem as a section of $C^p(\mathcal{H}^{\ast})\rightarrow B^{p+1}(\mathcal{H}^{\ast}),$ 
 we obtain 
\begin{equation}\label{eqw2a}
C^p(\mathcal{H}^{\ast})=s_{p}(B^{p+1}(\mathcal{H}^{\ast}))\oplus 
B^p(\mathcal{H}^{\ast}).
\end{equation}

 Let us consider the space $C^0(\mathcal{H}^{\ast})=H^0(\mathbb{S}^1;{\mathrm{Ad}_{\rho_{|_{\mathbb{S}^1}}}})$ in equation (\ref{eqw2a}). Since 
 $B^0(\mathcal{H}^{\ast})$ is trivial, we have
\begin{equation}\label{2eqnw1a}
C^0(\mathcal{H}^{\ast})=s_{0}(B^{1}(\mathcal{H}^{\ast}))\oplus 
B^0(\mathcal{H}^{\ast})=s_{0}(B^{1}(\mathcal{H}^{\ast})).
\end{equation}
By the exactness of sequence (\ref{long3a}), $B^{1}(\mathcal{H}^{\ast})=\mathrm{Im}(\alpha_0)$ is isomorphic to $H^0(\mathbb{S}^1;{\mathrm{Ad}_{\rho_{|_{\mathbb{S}^1}}}}).$ So, we can choose the basis $\mathbf{b}^1$ as $\alpha_0(\mathbf{h}^{0}_{\mathbb{S}^1}).$ Since the section $s_0$ is the inverse of $\alpha_0,$ by equation (\ref{2eqnw1a}), $s_{0}(\mathbf{b}^1)=\mathbf{h}^{0}_{\mathbb{S}^1}$ becomes the new basis for $C^0(\mathcal{H}^{\ast}).$ Hence, we obtain
\begin{equation}\label{reqnw2a}
[s_{0}(\mathbf{b}^1),\mathbf{h}_{\mathbb{S}^1}^0]=1. 
\end{equation}

We now consider the vector space $C^1(\mathcal{H}^{\ast})=H^1(\Sigma_{g,2};\mathrm{Ad}_\rho)$ in equation (\ref{eqw2a}). Then  we have
\begin{equation}\label{eqnw21a}
C^1(\mathcal{H}^{\ast})=s_{1}(B^{2}(\mathcal{H}^{\ast}))\oplus 
B^1(\mathcal{H}^{\ast}).
\end{equation}
Recall that $\mathrm{dim}(H^0(\mathbb{S}^1;{\mathrm{Ad}_{\rho_{|_{\mathbb{S}^1}}}}))=1$ and $\mathbf{h}^{1}_{\Sigma_{g,2}}=\left\{\mathbf{h}^{1,j}_{\Sigma_{g,2}}\right\}_{j=1}^{12k}$ is the given basis of $H^1(\Sigma_{g,2};\mathrm{Ad}_\rho).$ As $B^1(\mathcal{H}^{\ast})$ is a one-dimensional subspace of $C^1(\mathcal{H}^{\ast}),$ there is a non-zero vector 
$(a_{1,1},a_{1,2},\ldots,a_{1,2k})$ such that
\begin{equation}\label{basis1a}
\mathbf{b}^1=\left\{\sum_{j=1}^{12k}{a_{1,j} \; \mathbf{h}^{1,j}_{\Sigma_{g,2}}}
 \right\}.
\end{equation}
Since $s_{1}(B^{2}(\mathcal{H}^{\ast}))$ is a $(12k-1)$-dimensional subspace of $C^1(\mathcal{H}^{\ast}),$ there are non-zero vectors 
$(a_{i,1},a_{i,2},\ldots,a_{i,12k})$ for $i\in \{2,\ldots,12k\}$ such that
\begin{equation*}
\left\{\sum_{j=1}^{12k}{a_{i,j} \; \mathbf{h}^{1,j}_{\Sigma_{g,2}}}
 \right\}_{i=2}^{12k}
\end{equation*}
becomes a basis for $s_{1}(B^{2}(\mathcal{H}^{\ast}))$ and $A=[a_{ij}]$ is a non-singular matrix. Let us choose the basis of $s_{1}(B^{2}(\mathcal{H}^{\ast}))$ as 
\begin{equation*}
s_{1}(\mathbf{b}^2)=\left \{\det(A)^{-1}\left[\sum_{j=1}^{12k}{a_{2,j} \; \mathbf{h}^{1,j}_{\Sigma_{g,2}}}\right],\left\{\sum_{j=1}^{12k}{a_{i,j} \; \mathbf{h}^{1,j}_{\Sigma_{g,2}}}
 \right\}_{i=3}^{12k} \right\}.
\end{equation*}
From equation (\ref{eqnw21a}) it follows that
$s_{1}(\mathbf{b}^2)\sqcup \mathbf{b}^1$ is the new basis for 
$C^1(\mathcal{H}^{\ast}).$ Thus, we get
\begin{equation}\label{esrew1a}
[s_{1}(\mathbf{b}^2)\sqcup \mathbf{b}^1,\mathbf{h}^{1}_{\Sigma_{g,2}}]=1. 
\end{equation}

If we use equation (\ref{eqw2a}) for the space $C^2(\mathcal{H}^{\ast})=H^1(\Sigma_{k,2}^1;\mathrm{Ad}_{\rho_{|_{\Sigma_{k,2}^1}}})\oplus H^1(\Sigma_{k,2}^2;\mathrm{Ad}_{\rho_{|_{\Sigma_{k,2}^2}}}),$ then we have
\begin{equation}\label{neweq6a}
C^2(\mathcal{H}^{\ast})=s_{2}(B^{3}(\mathcal{H}^{\ast}))\oplus 
B^2(\mathcal{H}^{\ast}).
\end{equation} 
Since $B^{3}(\mathcal{H}^{\ast})=H^1(\mathbb{S}^1;{\mathrm{Ad}_{\rho_{|_{\mathbb{S}^1}}}}),$ we can choose the basis $\mathbf{b}^3$ as $\mathbf{h}^{1}_{\mathbb{S}^1}.$ Hence, 
$s_2(\mathbf{b}^3)=s_2(\mathbf{h}^{1}_{\mathbb{S}^1})$ becomes a basis of $s_{2}(B^{3}(\mathcal{H}^{\ast})).$ In the previous step, we chose the basis of $B^2(\mathcal{H}^{\ast})$ as follows
\begin{eqnarray*}
&&\mathbf{b}^2=\beta_0(s_{1}(\mathbf{b}^2))\\ \nonumber
&&\quad \; =\left \{ \det(A)^{-1}\left[\sum_{j=1}^{12k}{a_{2,j} \; \beta_0(\mathbf{h}^{1,j}_{\Sigma_{g,2}}})\right],\left\{\sum_{j=1}^{12k}{a_{i,j} \; \beta_0(\mathbf{h}^{1,j}_{\Sigma_{g,2}}})
 \right\}_{i=3}^{12k} \right\}.
\end{eqnarray*}
Let us denote the basis elements of $\mathbf{b}^2$ as $\mathbf{b}^{2,j}$ for 
$j\in \{1,\ldots,12k-1\}.$ From equation (\ref{neweq6a}) it follows that $s_{2}(\mathbf{b}^3)\sqcup \mathbf{b}^2$ becomes the new basis for $C^2(\mathcal{H}^{\ast}).$ As $H^1(\Sigma_{k,2}^1;\mathrm{Ad}_{\rho_{|_{\Sigma_{k,2}^1}}})$ and $H^1(\Sigma_{k,1}^2;\mathrm{Ad}_{\rho_{|_{\Sigma_{k,1}^2}}})$ are 
$(6k)$-dimensional subspaces of $(12k)$-dimensional space
$C^2(\mathcal{H}^{\ast}),$ there are non-zero vectors 
$(b_{{i,1}},\cdots,b_{{i,12k}})$ for $i\in \{1,\ldots,12k\}$ such that
 $$\left\{\sum_{j=1}^{12k-1}b_{i,j}\mathbf{b}^{2,j}
+ b_{i,12k}s_2(\mathbf{h}^{1}_{\mathbb{S}^{1}})\right\}_{i=1}^{6k}$$
 and 
$$\left\{\sum_{j=1}^{12k-1}b_{i,j}\mathbf{b}^{2,j}
+ b_{i,12k}s_2(\mathbf{h}^{1}_{\mathbb{S}^{1}})\right\}_{i=6k-1}^{12k}$$ 
are bases of
$H^1(\Sigma_{k,2}^1;\mathrm{Ad}_{\rho_{|_{\Sigma_{k,2}^1}}})$ and 
$H^1(\Sigma_{k,2}^2;\mathrm{Ad}_{\rho_{|_{\Sigma_{k,2}^2}}}),$ respectively. Moreover, the matrix $B=[b_{i,j}]$ has non-zero determinant. 
Let us take the bases of
$H^1(\Sigma_{k,2}^1;\mathrm{Ad}_{\rho_{|_{\Sigma_{k,2}^1}}})$ and 
$H^1(\Sigma_{k,2}^2;\mathrm{Ad}_{\rho_{|_{\Sigma_{k,2}^2}}})$ as follows

\begin{eqnarray*}
\mathbf{h}^1_{\Sigma_{k,2}^1}&=& \left\{\det(B)^{-1}\left[\sum_{j=1}^{12k-1}b_{1,j}\mathbf{b}^{2,j}
+ b_{1,12k}s_2(\mathbf{h}^{1}_{\mathbb{S}^{1}})\right],
\right.\\
&&  \; \; \; \left. \left\{\sum_{j=1}^{12k-1}b_{i,j}\mathbf{b}^{2,j}
+ b_{i,12k}s_2(\mathbf{h}^{1}_{\mathbb{S}^{1}})\right\}_{i=2}^{6k}\right\},\\
\mathbf{h}^1_{\Sigma_{k,2}^2}&=&\left\{\sum_{j=1}^{12k-1}b_{i,j}\mathbf{b}^{2,j}
+ b_{i,12k}s_2(\mathbf{h}^{1}_{\mathbb{S}^{1}})\right\}_{i=6k-1}^{12k}.
\end{eqnarray*}
Hence, $\mathbf{h}^1_{\Sigma_{k,2}^1}\sqcup \mathbf{h}^1_{\Sigma_{k,2}^2}$ is the initial basis for $C^2(\mathcal{H}^{\ast}),$ and thus we obtain
 \begin{equation}\label{eqwmol5a}
  [s_{2}(\mathbf{b}^3)\sqcup \mathbf{b}^2,\mathbf{h}^1_{\Sigma_{k,2}^1}\sqcup \mathbf{h}^1_{\Sigma_{k,2}^2}]=1.
\end{equation}
 
 We now consider equation (\ref{eqw2a}) for $C^3(\mathcal{H}^{\ast})= H^1(\mathbb{S}^1;{\mathrm{Ad}_{\rho_{|_{\mathbb{S}^1}}}}).$ Then we get
 \begin{equation}\label{eqnwry1a}
C^3(\mathcal{H}^{\ast})=s_{3}(B^{4}(\mathcal{H}^{\ast}))\oplus 
B^3(\mathcal{H}^{\ast})=B^3(\mathcal{H}^{\ast}).
\end{equation}
In the previous step, we chose the basis of $B^3(\mathcal{H}^{\ast})$ as  
$\mathbf{b}^3=\mathbf{h}^{1}_{\mathbb{S}^1}.$ Thus, by equation 
(\ref{eqnwry1a}), $\mathbf{b}^3$ becomes the new basis of $C^3(\mathcal{H}^{\ast})$. Since $\mathbf{h}^{1}_{\mathbb{S}^1}$ is also the given basis of $C^3(\mathcal{H}^{\ast}),$ we have
 \begin{equation}\label{eqwm2a}
  [\mathbf{b}^3,\mathbf{h}^{1}_{\mathbb{S}^1}]=1.
\end{equation}

Combining equations (\ref{reqnw2a}), (\ref{esrew1a}),  
(\ref{eqwmol5a}), and (\ref{eqwm2a}), the corrective term becomes one as follows
\begin{equation}\label{lonh4a}
\mathbb{T}\left( \mathcal{H}^{\ast},\{\mathbf{h}^p\}_{p=0}^3,\{0\}_{p=0}^3\right)
=\prod_{p=0}^3 \left[s_p(\mathbf{b}^{p+1})\sqcup
\mathbf{b}^p, \mathbf{h}^p\right]^{(-1)^{(p+1)}}=1.
\end{equation} 

Theorem \ref{prt1} and equation (\ref{lonh4a}) finish the proof of Theorem \ref{mainth1}.
\end{proof}

\begin{theorem}\label{mainth2} 
Let $[\rho] \in \mathcal{R}^{g}(\pi_1(\Sigma_{g,1}),\partial(\Sigma_{g,1}),G)_{\rho_0}$ and let $\rho$ satisfy condition $C_2$ for the connected sum $\Sigma_{g,1}=\Sigma_{k,1} \# \Sigma_{k,0}.$ For given bases $\mathbf{h}^1_{\Sigma_{g,1}}$ and $\mathbf{h}^{i}_{\mathbb{S}^1}$ of $H^1(\Sigma_{g,1};\mathrm{Ad}_\rho)$ and $H^i(\mathbb{S}^1;{\mathrm{Ad}_{\rho_{|_{\mathbb{S}^1}}}}),$ $i=0,1,$ there exist respectively homology bases $\mathbf{h}^1_{\Sigma_{k,2}}$ and $\mathbf{h}^1_{\Sigma_{k,1}}$ of $H^1(\Sigma_{k,2};\mathrm{Ad}_{\rho_{|_{\Sigma_{k,2}}}})$ and $H^1(\Sigma_{k,1};\mathrm{Ad}_{\rho_{|_{\Sigma_{k,1}}}})$ such that the corrective term becomes one and the 
following formula is valid
\begin{eqnarray*}
 \mathbb{T}(\Sigma_{g,1};\rho,\{0,\mathbf{h}^{1}_{\Sigma_{g,1}},0\})&=&\mathbb{T}(\Sigma_{k,2};{\rho_{|_{\Sigma_{k,2}}}},\{0,\mathbf{h}^{1}_{\Sigma_{k,2}},0\}) \\
& & \times\;\mathbb{T}(\Sigma_{k,1};{\rho_{|_{\Sigma_{k,1}}}},\{0,\mathbf{h}^{1}_{\Sigma_{k,1}},0\}) \\
& & \times\;{\mathbb{T}(\mathbb{S}^1;\rho_{|_{\mathbb{S}^1}},\{\mathbf{h}^{0}_{\mathbb{S}^1},\mathbf{h}^{1}_{\mathbb{S}^1}\})}^{-1}.
\end{eqnarray*}
\end{theorem}

\begin{proof}
By condition $C_2,$ we obtain good and $\partial$-regular restrictions $\rho_{|_{\pi_1(\Sigma_{k,2})}},$ $\rho_{|_{\pi_1(\Sigma_{k,1})}}.$ Following 
Remark~\ref{rem:bndry1}, we conclude that the restriction $\rho_{|_{\mathbb{S}^1}}$ is regular. Consider the following short exact sequence of the twisted cochain complexes
\begin{eqnarray*}\label{seq3}
&& 0\to C^{\ast}(\Sigma_{g,1};\mathrm{Ad}_\rho) \rightarrow 
C^{\ast}(\Sigma_{k,2};\mathrm{Ad}_{\rho_{|_{\Sigma_{k,2}}}})\oplus
C^{\ast}(\Sigma_{k,1};\mathrm{Ad}_{\rho_{|_{\Sigma_{k,1}}}})\\
 &&\quad \quad \quad \quad \quad \; \tikz\draw[->,rounded corners](5.50,0.4)--(5.50,0)--(1,0)--(1,-0.4) node[yshift=4.5ex,xshift=15.0ex] {}; \nonumber\\ 
&&\quad \quad C^{\ast}(\mathbb{S}^1;{\mathrm{Ad}_{\rho_{|_{\mathbb{S}^1}}}}) \to 0.
\end{eqnarray*}

By Theorem \ref{gold1}, the Mayer-Vietoris long exact sequence associated with the above short exact sequence becomes
\begin{eqnarray}\label{long3}
 \nonumber &\mathcal{H}^{\ast}:& 0\longrightarrow H^0(\mathbb{S}^1;{\mathrm{Ad}_{\rho_{|_{\mathbb{S}^1}}}}) \stackrel{\alpha_0}{\longrightarrow} H^1(\Sigma_{g,1};\mathrm{Ad}_\rho)\\
 &&\quad \quad \quad \quad \quad \quad \quad \quad \;\; \tikz\draw[->,rounded corners](3.25,0.4)--(3.25,0)--(1,0)--(1,-0.4) node[yshift=4.5ex,xshift=8.0ex] {$\beta_0$}; \nonumber\\ 
 && H^1(\Sigma_{k,2};\mathrm{Ad}_{\rho_{|_{\Sigma_{k,2}}}})\oplus H^1(\Sigma_{k,1};\mathrm{Ad}_{\rho_{|_{\Sigma_{k,1}}}}) \nonumber\\
 &&\quad \quad \quad   \tikz\draw[->,rounded corners](2.90,0.4)--(2.90,0)--(1,0)--(1,-0.4) node[yshift=4.5ex,xshift=7.0ex] {$\partial$}; \nonumber\\ 
&& H^1(\mathbb{S}^1;{\mathrm{Ad}_{\rho_{|_{\mathbb{S}^1}}}}) \stackrel{\alpha_1} {\longrightarrow} 0 .
  \end{eqnarray}

By using the same arguments stated in the proof of Theorem \ref{mainth1}, we conclude that the corrective term becomes one. Precisely,
\begin{equation}\label{lonh4}
\mathbb{T}\left( \mathcal{H}^{\ast},\{\mathbf{h}^p\}_{p=0}^3,\{0\}_{p=0}^3\right)
=\prod_{p=0}^3 \left[s_p(\mathbf{b}^{p+1})\sqcup
\mathbf{b}^p, \mathbf{h}^p\right]^{(-1)^{(p+1)}}=1.
\end{equation} 
Theorem \ref{prt1} and equation (\ref{lonh4}) complete the proof of Theorem \ref{mainth2}.
\end{proof}

Note that the basis $\mathbf{h}^{\ast}_{\mathbb{S}^1}$ of $H^{\ast}(\mathbb{S}^1;{\mathrm{Ad}_{\rho_{|_{\mathbb{S}^1}}}})$ is arbitrary in Theorem \ref{mainth1} and Theorem \ref{mainth2}. Let us fix these bases as $\mathbf{h}^{1}_{\mathbb{S}^1}$ and $\mathbf{h}^{0}_{\mathbb{S}^1}$ for $H^{1}(\mathbb{S}^1;{\mathrm{Ad}_{\rho_{|_{\mathbb{S}^1}}}})$ and $H^{0}(\mathbb{S}^1;{\mathrm{Ad}_{\rho_{|_{\mathbb{S}^1}}}}).$ Then all the twisted torsions of circles 
$\mathbb{S}^1$ with these fix bases are equal. If we use Theorem \ref{mainth1} and Theorem \ref{mainth2} inductively, we get the following result.

\begin{theorem}\label{mainthm3}Let $[\rho] \in \mathcal{R}^{g}(\pi_1(\Sigma_{g,2}),\partial(\Sigma_{g,2}),G)_{\rho_0}$ and let $\rho$ satisfy condition $C_2.$ Fix the basis of $H^j(\mathbb{S}^1;{\mathrm{Ad}_{\rho_{|_{\mathbb{S}^1}}}})$ as $\mathbf{h}^{j}_{\mathbb{S}^1}$ for $j=0,1.$ For a given basis $\mathbf{h}^1_{\Sigma_{g,2}}$ of $H^1(\Sigma_{g,2};\mathrm{Ad}_\rho),$ there exist bases $\mathbf{h}^1_{\Sigma_{2,2}^i}$ of $H^1(\Sigma_{2,2}^i;\mathrm{Ad}_{\rho_{|_{\Sigma_{2,2}^i}}})$ for each $i \in \{1,\ldots,k\}$ such that the following multiplicative gluing formula holds
\begin{eqnarray*}
 \mathbb{T}(\Sigma_{g,2};\rho,\{0,\mathbf{h}^{1}_{\Sigma_{g,2}},0\})&=& 
\prod_{i=1}^{k}
\mathbb{T}(\Sigma_{2,2}^i;{\rho_{|_{\Sigma_{2,2}^i}}},\{0,\mathbf{h}^{1}_{\Sigma_{2,2}^i},0\})\\
&& \times\; {\mathbb{T}(\mathbb{S}^1;\rho_{|_{\mathbb{S}^1}},
\{\mathbf{h}^{0}_{\mathbb{S}^1},\mathbf{h}^{1}_{\mathbb{S}^1}\})}^{-k+1}.
\end{eqnarray*}
\end{theorem}

If we use Theorem \ref{mainth2} and Theorem \ref{mainthm3} inductively, we obtain the following result.

\begin{theorem}\label{mainthm4}
Let $[\rho] \in \mathcal{R}^{g}(\pi_1(\Sigma_{g,1}),\partial(\Sigma_{g,1}),G)_{\rho_0}$ and let $\rho$ satisfy condition $C_2.$ For a given basis $\mathbf{h}^1_{\Sigma_{g,1}}$ and a fixed basis $\mathbf{h}^{j}_{\mathbb{S}^1}$ of $H^1(\Sigma_{g,1};\mathrm{Ad}_\rho)$ and $H^j(\mathbb{S}^1;{\mathrm{Ad}_{\rho_{|_{\mathbb{S}^1}}}}),$ $j=0,1,$ there exist respectively bases $\mathbf{h}^1_{\Sigma_{2,2}^i}$ and $\mathbf{h}^1_{\Sigma_{2,1}}$ of $H^1(\Sigma_{2,2}^i;\mathrm{Ad}_{\rho_{|_{\Sigma_{2,2}^i}}})$ and $H^1(\Sigma_{2,1};\mathrm{Ad}_{\rho_{|_{\Sigma_{2,1}}}})$ for each $i \in \{1,\ldots,k-1\}$ such that the corrective term becomes one and the 
following formula is valid
\begin{eqnarray*}
 \mathbb{T}(\Sigma_{g,1};\rho,\{0,\mathbf{h}^{1}_{\Sigma_{g,1}},0\})&=& 
\prod_{i=1}^{k-1}
\mathbb{T}(\Sigma_{2,2}^i;{\rho_{|_{\Sigma_{2,2}^i}}},\{0,\mathbf{h}^{1}_{\Sigma_{2,2}^i},0\})\\
&& \times\;
\mathbb{T}(\Sigma_{2,1};{\rho_{|_{\Sigma_{2,1}}}},\{0,\mathbf{h}^{1}_{\Sigma_{2,1}},0\})\\
&& \times\; {\mathbb{T}(\mathbb{S}^1;\rho_{|_{\mathbb{S}^1}},
\{\mathbf{h}^{0}_{\mathbb{S}^1},\mathbf{h}^{1}_{\mathbb{S}^1}\})}^{-k+1}.
\end{eqnarray*}
\end{theorem}

Now Theorem \ref{thm:mainthm} follows directly from equation (\ref{eq2mn1}) in Theorem \ref{mainth1} and Theorem \ref{mainthm4}.

Let us consider the decomposition 
$\Sigma_{g+2,n}=\Sigma_{g,0}{\#}\Sigma_{2,n}$ for any $n\geq 1$ and assume that $\rho:\pi_1(\Sigma_{g+2,n})\rightarrow G$ is a good and $\partial$-regular representation such that the restrictions $\rho_{|_{\pi_1(\Sigma_{g,1})}}$ and 
$\rho_{|_{\pi_1(\Sigma_{2,n+1})}}$ are good and $\partial$-regular as well. By Remark \ref{rem:bndry1}, the restriction $\rho_{|_{\mathbb{S}^1}}$ is regular. Following the same arguments stated in the proof of Theorem \ref{mainth1}, we have:
\begin{theorem}\label{mainthm5}
Let $[\rho] \in \mathcal{R}^{g}(\pi_1(\Sigma_{g+2,n}),\partial(\Sigma_{g+2,n}),G)_{\rho_0}$ and let $\rho$ satisfy the above assumption. For a given basis $\mathbf{h}^1_{\Sigma_{g+2,n}}$ of $H^1(\Sigma_{g+2,n};\mathrm{Ad}_\rho)$ and a fixed basis $\mathbf{h}^{j}_{\mathbb{S}^1}$ of $H^j(\mathbb{S}^1;{\mathrm{Ad}_{\rho_{|_{\mathbb{S}^1}}}}),$ $j=0,1,$ there exist respectively bases $\mathbf{h}^1_{\Sigma_{g,1}}$ and 
$\mathbf{h}^1_{\Sigma_{2,n+1}}$ of $H^1(\Sigma_{g,1};\mathrm{Ad}_{\rho_{|_{\Sigma_{g,1}}}})$ and $H^1(\Sigma_{2,n+1};\mathrm{Ad}_{\rho_{|_{\Sigma_{2,n+1}}}})$ such that the
following gluing formula holds
\begin{eqnarray*}
 \mathbb{T}(\Sigma_{g+2,n};\rho,\{0,\mathbf{h}^{1}_{\Sigma_{g+2,n}},0\})&=& 
\mathbb{T}(\Sigma_{g,1};{\rho_{|_{\Sigma_{g,1}}}},\{0,\mathbf{h}^{1}_{\Sigma_{g,1}},0\})\\
&& \times\;
\mathbb{T}(\Sigma_{2,n+1};{\rho_{|_{\Sigma_{2,n+1}}}},\{0,\mathbf{h}^{1}_{\Sigma_{2,n+1}},0\})\\
&& \times\; {\mathbb{T}(\mathbb{S}^1;\rho_{|_{\mathbb{S}^1}},
\{\mathbf{h}^{0}_{\mathbb{S}^1},\mathbf{h}^{1}_{\mathbb{S}^1}\})}^{-1}.
\end{eqnarray*}
Moreover, by Theorem \ref{mainthm4}, there exist respectively bases $\mathbf{h}^1_{\Sigma_{2,2}^i}$ and $\mathbf{h}^1_{\Sigma_{2,1}}$ of $H^1(\Sigma_{2,2}^i;\mathrm{Ad}_{\rho_{|_{\Sigma_{2,2}^i}}})$ and $H^1(\Sigma_{2,1};\mathrm{Ad}_{\rho_{|_{\Sigma_{2,1}}}})$ for each $i \in \{1,\ldots,k-1\}$ such that the
following formula is valid
\begin{eqnarray*}
  \mathbb{T}(\Sigma_{g+2,n};\rho,\{0,\mathbf{h}^{1}_{\Sigma_{g+2,n}},0\})&=& 
\prod_{i=1}^{k-1}
\mathbb{T}(\Sigma_{2,2}^i;{\rho_{|_{\Sigma_{2,2}^i}}},\{0,\mathbf{h}^{1}_{\Sigma_{2,2}^i},0\})\\
&& \times\;
\mathbb{T}(\Sigma_{2,1};{\rho_{|_{\Sigma_{2,1}}}},\{0,\mathbf{h}^{1}_{\Sigma_{2,1}},0\})\\
&& \times\;
\mathbb{T}(\Sigma_{2,n+1};{\rho_{|_{\Sigma_{2,n+1}}}},\{0,\mathbf{h}^{1}_{\Sigma_{2,n+1}},0\})\\
&& \times\; {\mathbb{T}(\mathbb{S}^1;\rho_{|_{\mathbb{S}^1}},
\{\mathbf{h}^{0}_{\mathbb{S}^1},\mathbf{h}^{1}_{\mathbb{S}^1}\})}^{-k}.
\end{eqnarray*}
\end{theorem}

\subsection{Further Consequences}
\label{sec:9}

The following is the proof of Corollary \ref{cor:closed}.
\begin{proof}
From Definition \ref{defcrl1} it follows that
\begin{equation}\label{fn0}
{\mathbb{T}(\mathbb{S}^1;\rho_{|_{\mathbb{S}^1}},
\{\mathbf{h}^{0}_{\mathbb{S}^1},\mathbf{h}^{1}_{\mathbb{S}^1}\})}=\frac{\nu(\wedge \mathbf{h}^{1}_{\mathbb{S}^1})^2}{<\wedge\mathbf{h}^{1}_{\mathbb{S}^1},\wedge\mathbf{h}^0_{\mathbb{S}^1}>^2}.
\end{equation}
Since $G =SL_2(\mathbb{C})$ is simply connected, it has rank $r=2-1=1$ and dimension $d=2^2-1=3.$ Thus, 
for $g=2^t\geq 4$ and $k=g/2$ we have
$$\mathrm{dim}(H^1(\Sigma_{g,0};\mathrm{Ad}_\rho))=-\chi(\Sigma_{g,0})\;d=6g-6=12k-6.$$

By Theorem \ref{wittheo}, the following equation holds
\begin{equation}\label{fn1}
\left|\mathbb{T}(\Sigma_{g,0},\{0,\mathbf{h}^1_{\Sigma_{g,0}},0\})\right|= \frac{\left|\omega_{_{\Sigma_{g,0}}}^{6k-3}(\wedge \mathbf{h}^1_{\Sigma_{g,0}})\right|} {(6k-3)!}.
\end{equation}

By Theorem \ref{prtitheo}, we obtain the following formulas
\begin{equation}\label{fn2}
\left|\mathbb{T}(\Sigma_{2,2},\{0,\mathbf{h}^1_{\Sigma_{2,2}},0\})\right|= \frac{\left|\omega_{\Sigma_{2,2}}^5(\wedge \mathbf{h}^1_{\Sigma_{2,2}})\wedge \nu_1^*  \wedge \nu_2^* \right|} {5!}.
\end{equation}

\begin{equation}\label{fn3}
\left|\mathbb{T}(\Sigma_{2,1},\{0,\mathbf{h}^1_{\Sigma_{2,1}},0\})\right|= \frac{\left|\omega_{\Sigma_{2,1}}^4(\wedge \mathbf{h}^1_{\Sigma_{2,1}})\wedge \nu_1^* \right|} {4!}.
\end{equation}

Combining Theorem \ref{thm:mainthm} and equations (\ref{fn0})-(\ref{fn3}) finish the proof of Corollary~\ref{cor:closed}. Namely, we have
\begin{eqnarray*}
 \left|\omega_{\Sigma_{g,0}}^{6k-3}(\wedge \mathbf{h}^1_{\Sigma_{g,0}})\right|&=& 
\mathcal{M}_k\; \prod_{i=1}^{k-2}
\left|\omega_{\Sigma_{2,2}^i}^5(\wedge \mathbf{h}^1_{\Sigma_{2,2}^i})\wedge \nu_1^*  \wedge \nu_2^* \right| \; \prod_{\mu=1}^{2}
\left|\omega_{\Sigma_{2,1}^{\mu}}^4(\wedge \mathbf{h}^1_{\Sigma_{2,1}^{\mu}})\wedge \nu_1^* \right|.
\end{eqnarray*}
Here, $\mathcal{M}_k=(6k-3)!\; (5!)^{2-k}\; (4!)^{-2}  \; {\frac{\nu(\wedge \mathbf{h}^{1}_{\mathbb{S}^1})^{-2k+2}}{<\wedge\mathbf{h}^{1}_{\mathbb{S}^1},\wedge\mathbf{h}^0_{\mathbb{S}^1}>^{-2k+2}}}\in \mathbb{C}.$
\end{proof}

The following corollaries are obtained similarly.
\begin{corollary}
Let  $[\rho] \in \mathcal{R}^{g}(\pi_1(\Sigma_{g,2}),\partial(\Sigma_{g,2}),G)_{\rho_0}$ and let $\rho$ satisfy condition $C_2.$ Fix the basis of $H^j(\mathbb{S}^1;{\mathrm{Ad}_{\rho_{|_{\mathbb{S}^1}}}})$ as $\mathbf{h}^{j}_{\mathbb{S}^1}$ for $j=0,1.$ For a given basis $\mathbf{h}^1_{\Sigma_{g,2}}$ of $H^1(\Sigma_{g,2};\mathrm{Ad}_\rho),$ there exists a basis $\mathbf{h}^1_{\Sigma_{2,2}^i}$ of $H^1(\Sigma_{2,2}^i;\mathrm{Ad}_{\rho_{|_{\Sigma_{2,2}^i}}})$ for each $i \in \{1,\ldots,k\}$ such that the following formula holds
\begin{eqnarray*}
\left|\omega_{\Sigma_{g,2}}^{6k-1}(\wedge \mathbf{h}^1_{\Sigma_{g,2}})\right|&=& 
\mathcal{M}_k\; \prod_{i=1}^{k}
\left|\omega_{\Sigma_{2,2}^i}^5(\wedge \mathbf{h}^1_{\Sigma_{2,2}^i})\wedge 
\nu_1^* \wedge \nu_2^* \right|.
\end{eqnarray*}
Here, $\mathcal{M}_k=(6k-1)!\; (5!)^{-k} \; {\frac{\nu(\wedge \mathbf{h}^{1}_{\mathbb{S}^1})^{-2k+2}}{<\wedge\mathbf{h}^{1}_{\mathbb{S}^1},\wedge\mathbf{h}^0_{\mathbb{S}^1}>^{-2k+2}}}\in \mathbb{C}.$
\end{corollary}

\begin{corollary}
Let $[\rho] \in \mathcal{R}^{g}(\pi_1(\Sigma_{g,1}),\partial(\Sigma_{g,1}),G)_{\rho_0}$ and let $\rho$ satisfy condition $C_2.$ For a given basis $\mathbf{h}^1_{\Sigma_{g,1}}$ of $H^1(\Sigma_{g,1};\mathrm{Ad}_\rho)$ and a fixed basis $\mathbf{h}^{j}_{\mathbb{S}^1}$ of $H^j(\mathbb{S}^1;{\mathrm{Ad}_{\rho_{|_{\mathbb{S}^1}}}}),$ $j=0,1,$ there are bases $\mathbf{h}^1_{\Sigma_{2,2}^i}$ and $\mathbf{h}^1_{\Sigma_{2,1}}$ of $H^1(\Sigma_{2,2}^i;\mathrm{Ad}_{\rho_{|_{\Sigma_{2,2}^i}}})$ and $H^1(\Sigma_{2,1};\mathrm{Ad}_{\rho_{|_{\Sigma_{2,1}}}}),$ respectively  for each $i \in \{1,\ldots,k-1\}$ such that the 
following formula is valid for the volume forms
\begin{eqnarray*}
 \left|\omega_{\Sigma_{g,1}}^{6k-2}(\wedge \mathbf{h}^1_{\Sigma_{g,1}})\right|
&=& \mathcal{M}_k\; \prod_{i=1}^{k-1}
\left|\omega_{\Sigma_{2,2}^i}^5(\wedge \mathbf{h}^1_{\Sigma_{2,2}^i})\wedge \nu_1^*  \wedge \nu_2^* \right| \;
\left|\omega_{\Sigma_{2,1}}^4(\wedge \mathbf{h}^1_{\Sigma_{2,1}})\wedge \nu_1^* \right|.
\end{eqnarray*}
Here, $\mathcal{M}_k=(6k-2)!\; (5!)^{1-k} \; (4!)^{-1} \; {\frac{\nu(\wedge \mathbf{h}^{1}_{\mathbb{S}^1})^{-2k+2}}{<\wedge\mathbf{h}^{1}_{\mathbb{S}^1},\wedge\mathbf{h}^0_{\mathbb{S}^1}>^{-2k+2}}}\in \mathbb{C}.$
\end{corollary}

\begin{corollary}
Let  $[\rho] \in \mathcal{R}^{g}(\pi_1(\Sigma_{g+2,n}),\partial(\Sigma_{g+2,n}),G)_{\rho_0}$ and let $\rho$ satisfy the assumption in Theorem~\ref{mainthm5}. For a given basis $\mathbf{h}^1_{\Sigma_{g+2,n}}$ of $H^1(\Sigma_{g+2,n};\mathrm{Ad}_\rho)$ and a fixed basis $\mathbf{h}^{j}_{\mathbb{S}^1}$ of $H^j(\mathbb{S}^1;{\mathrm{Ad}_{\rho_{|_{\mathbb{S}^1}}}}),$ $j=0,1,$ there exist bases $\mathbf{h}^1_{\Sigma_{2,2}^i},$ $\mathbf{h}^1_{\Sigma_{2,1}}$ and $\mathbf{h}^1_{\Sigma_{2,n+1}}$ of $H^1(\Sigma_{2,2}^i;\mathrm{Ad}_{\rho_{|_{\Sigma_{2,2}^i}}}),$ $H^1(\Sigma_{2,1};\mathrm{Ad}_{\rho_{|_{\Sigma_{2,1}}}})$ and $H^1(\Sigma_{2,n+1};\mathrm{Ad}_{\rho_{|_{\Sigma_{2,n+1}}}}),$ respectively for each $i \in \{1,\ldots,k-1\}$ such that the volume forms satisfy the following formula 
\begin{eqnarray*}
 \left|\omega_{\Sigma_{g+2,n}}^{6k+n+3}(\wedge \mathbf{h}^1_{\Sigma_{g+2,n}})\right|
&=& \mathcal{M}_k\; \prod_{i=1}^{k-1}
\left|\omega_{\Sigma_{2,2}^i}^5(\wedge \mathbf{h}^1_{\Sigma_{2,2}^i})\wedge \nu_1^*  \wedge \nu_2^* \right| 
\left|\omega_{\Sigma_{2,1}}^4(\wedge \mathbf{h}^1_{\Sigma_{2,1}})\wedge \nu_1^* \right|\\
& &\times \left|\omega_{\Sigma_{2,n+1}}^{4+n}(\wedge \mathbf{h}^1_{\Sigma_{2,n+1}})\wedge \nu_1^*\wedge\cdots \wedge \nu_{n+1}^* \right| .
\end{eqnarray*}
Here, $\mathcal{M}_k=(6k+n+3)!\; (5!)^{1-k}\; (4!)^{-1}  \; ((4+n)!)^{-1}\; {\frac{\nu(\wedge \mathbf{h}^{1}_{\mathbb{S}^1})^{-2k}}{<\wedge\mathbf{h}^{1}_{\mathbb{S}^1},\wedge\mathbf{h}^0_{\mathbb{S}^1}>^{-2k}}}\in \mathbb{C}.$
\end{corollary}



\end{document}